\documentclass[12pt]{extarticle}
\usepackage{amsmath, amsthm, amssymb, mathtools, hyperref,tikz}
\usepackage{graphicx}
\usepackage{longtable}
\oddsidemargin \evensidemargin
\newtheorem{theorem}{Theorem}
\numberwithin{theorem}{section}
\newtheorem{proposition}[theorem]{Proposition}
\newtheorem{lemma}[theorem]{Lemma}

\newtheorem{definition}[theorem]{Definition}
\newtheorem{remark}[theorem]{Remark}

\newtheorem{question}[theorem]{Question}

\newtheorem{example}[theorem]{Example}
\newtheorem{algorithm}[theorem]{Algorithm}

\newcommand{\RR}{\mathbb{R}}

\newcommand{\PP}{\mathbb{P}}

\newcommand{\ZZ}{\mathbb{Z}}

\DeclareMathOperator{\Del}{Del}

 \date{}
 
\title{\textbf{Schottky Algorithms: \\ Classical  meets Tropical}}
\author{Lynn Chua, Mario Kummer and Bernd Sturmfels}

\begin{document}
\maketitle

\begin{abstract} \noindent
We present a new perspective on the Schottky problem that links numerical computing
with tropical geometry. The task is to decide whether a 
symmetric matrix defines a Jacobian, and, if so, to compute the curve
and its canonical embedding. We offer solutions and their implementations in
genus four, both classically and tropically. 
The locus of cographic matroids arises from tropicalizing
the Schottky--Igusa modular form.
\end{abstract}

\section{Introduction}

The \textit{Schottky problem} \cite{gru} concerns the characterization of Jacobians of genus $g$ curves
 among all abelian varieties of dimension $g$. The latter are parametrized by the
 \textit{Siegel upper-half space} $\mathfrak{H}_g$, i.e.~the set of complex symmetric $g\times g$ 
 matrices $\tau$ with positive definite imaginary part. The {\em Schottky locus}
 $ \mathfrak{J}_g$ is the subset of matrices $\tau$ in $\mathfrak{H}_g$ that represent Jacobians.
 Both sets are complex analytic spaces whose dimensions reveal that the inclusion is proper for $g \geq 4$:
\begin{equation}
\label{eq:dimensions}
  {\rm dim}(\mathfrak{J}_g) \,= \, 3g -3 \qquad \hbox{and} \qquad
 {\rm dim}(\mathfrak{H}_g) \,= \, \binom{g+1}{2}. 
 \end{equation}
 For $g=4$, the dimensions in (\ref{eq:dimensions}) are $9$ and $10$,
 so $\mathfrak{J}_4$ is an analytic hypersurface in $\mathfrak{H}_4$.
 The equation defining this hypersurface is a polynomial of degree $16$ in
 the theta constants. First constructed by Schottky \cite{Schottky1888},
and further developed by Igusa \cite{igusa}, this modular form
embodies the theoretical solution (cf.~\cite[\S 3]{gru}) to the classical Schottky problem for $g=4$.

The Schottky problem also exists in tropical geometry \cite{mz}.
The {\em tropical Siegel space} $\mathfrak{H}_g^{\rm trop}$ is the 
cone of positive definite $g \times g$-matrices, endowed with the fan structure given by the
second Voronoi decomposition. The {\em tropical Schottky locus}
$\mathfrak{J}_g^{\rm trop}$ is the subfan indexed by cographic matroids \cite[Theorem 5.2.4]{BMV}.
A detailed analysis for $g\leq 5$ is found in \cite[Theorem 6.4]{chan12}.
It is known, e.g.~by~\cite[\S 6.3]{BMV}, that the inclusion 
$\mathfrak{J}_g^{\rm trop} \subset \mathfrak{H}_g^{\rm trop}$ correctly
tropicalizes the complex-analytic inclusion
$\mathfrak{J}_g \subset \mathfrak{H}_g$. However,
 it has been an open~problem (suggested in \cite[\S 9]{RSSS}) to find
a direct link between the equations that govern these two inclusions.

We here solve this problem, and 
develop computational tools for the Schottky problem,
both classically and tropically.
We distinguish between the {\em Schottky Decision Problem}
and the {\em Schottky Recovery Problem}. For the former, the input 
is a matrix $\tau$ in $\mathfrak{H}_g$ resp.~$\mathfrak{H}_g^{\rm trop}$,
possibly depending on parameters, and we must decide whether
$\tau$ lies in $\mathfrak{J}_g$ resp.~$\mathfrak{J}_g^{\rm trop}$.
For the latter, $\tau$ already passed that test, and we 
compute a curve whose Jacobian is given by~$\tau$.
The recovery problem also makes sense for $g=3$,
both classically \cite{bitangents} and tropically~\cite[\S 7]{BBC}.

This paper is organized as follows. In Section \ref{sec:classic} we tackle the
classical Schottky problem as a task in
numerical algebraic geometry \cite{DvH, HS, SD}. For $g=4$,
we utilize the software {\tt abelfunctions} \cite{abelfunctions} to
test whether the Schottky--Igusa modular form vanishes. In the affirmative case, 
we use a numerical version of Kempf's method \cite{kempf} to compute a
canonical embedding into $\PP^3$.
Our main results in Section \ref{sec:tropical}  are Algorithms \ref{alg:TSD} and \ref{alg:TSR}.
Based on the work in \cite{sikiric, sgsw, frank, voronoi},
these furnish a  computational solution to the tropical Schottky problem.
Key ingredients are  cographic matroids and the f-vectors of Voronoi polytopes.

Section \ref{sec:four} links the classical and tropical Schottky scenarios.
Theorem \ref{thm:recoverfromtheta} expresses the edge lengths of a 
metric graph in terms of tropical theta constants, and
Theorem~\ref{thm:admissible} explains what happens to the
Schottky--Igusa modular form in the tropical limit.
We found it especially gratifying to discover how the cographic locus
is encoded in the classical theory.

\smallskip

The software we describe in this paper is made available at the supplementary website
\begin{equation}
\label{eq:url}
 \hbox{\url{http://eecs.berkeley.edu/~chualynn/schottky}}
 \end{equation}
This contains several pieces of code for the tropical Schottky problem,
as well as a more coherent \texttt{Sage} program 
for the classical Schottky problem that makes calls to
{\tt abelfunctions}.

\section{The Classical Schottky Problem}\label{sec:classic}

We fix $g=4$, and review theta functions and 
Igusa's construction \cite{igusa} of the equation that
cuts out~$\mathfrak{J}_4$. 
For any vector $m\in\mathbb{Z}^{8}$ we write
 $m= (m',m'') $ for suitable $m',m''\in\mathbb{Z}^4$.
  The \textit{Riemann theta function} 
with characteristic $m$~is the following function of  $\tau\in\mathfrak{H}_4$ and
$z\in\mathbb{C}^4$:
\begin{align}\label{eqn:thetafunc}
 \theta[m](\tau,z)\,\,=\,\,\,\sum_{n\in\mathbb{Z}^4}\exp\left[\pi \textnormal{i} (n+\frac{m'}{2})^t \tau (n+\frac{m'}{2})+2\pi\textnormal{i}(n+\frac{m'}{2})^t(z+\frac{m''}{2})\right].
\end{align}
For numerical computations of the theta function one has to make a good choice of lattice points to sum over in order for this series to converge rapidly  \cite{DHBvHS, SD}.
We use the software {\tt abelfunctions} \cite{abelfunctions}
to evaluate $\theta[m]$ for arguments $\tau$ and $z$
with floating point coordinates.

Up to a global multiplicative factor,
the definition (\ref{eqn:thetafunc}) depends only on the image of $m$ in
$ (\mathbb{Z}/2\mathbb{Z})^8$.
The sign of the characteristic $m$ is $\,e(m)=(-1)^{(m')^t m''}$.
Namely, $m$ is \textit{even} if $e(m)=1$ and \textit{odd} if $e(m) = -1$.
A triple $\{m_1,m_2,m_3\}\subset (\mathbb{Z}/2\mathbb{Z})^8$
 is called \textit{azygetic} if $e(m_1)e(m_2)e(m_3)e(m_1+m_2+m_3) = -1$.
Suppose that this holds. Then we choose a rank $3$ subgroup $N$ of $
(\mathbb{Z}/2\mathbb{Z})^8$ such that all elements of
 $(m_1{+}N) \cup (m_2{+}N) \cup (m_3{+}N)$ are even.
 
We  consider the following three products of eight {\em theta constants} each:
\begin{equation}
\label{eq:pidef}
  \pi_i \,\,\,\,=\prod_{m\in m_i+N} \! \theta[m](\tau,0) \qquad
  \hbox{for} \quad i=1,2,3.
\end{equation}
 
\begin{theorem}[Igusa \cite{igusa}]\label{thm:igusa}
 The function $\mathfrak{H}_4 \rightarrow \mathbb{C}$ that takes
 a symmetric $4 {\times} 4$-matrix $\tau$ to
\begin{equation}
\label{eq:schottkyigusa} \pi_1^2+\pi_2^2+\pi_3^2-2\pi_1\pi_2-2 \pi_1\pi_3-2 \pi_2\pi_3
\end{equation}
is independent of the choices above. It vanishes
if and only if $\tau $ lies in the closure of the Schottky locus $\mathfrak{J}_4$.
\end{theorem}

We refer to the expression (\ref{eq:schottkyigusa}) as the {\em Schottky--Igusa modular form}.
This is a polynomial of degree $16$ in the theta constants  $\theta[m](\tau,0)$. Of course, the
formula is unique only modulo the ideal that defines the embedding of the moduli space $\mathcal{A}_4$ in
the $\PP^{15}$ of theta constants. 

Our implementation uses the polynomial that is given by the following specific choices:
$$
 m_1=\begin{pmatrix}1 & 1 \\ 0 & 0 \\ 1 & 1 \\ 0 & 0     \end{pmatrix}\!,\,
m_2=\begin{pmatrix} 0 & 1 \\ 0 & 0 \\ 0 & 0 \\ 1 & 0    \end{pmatrix}\!,\,
m_3=\begin{pmatrix} 0 & 1 \\ 0 & 0 \\ 1 & 1 \\ 1 &  1  \end{pmatrix}\!,\,\,
  n_1=\begin{pmatrix} 0 & 1 \\ 0 & 1 \\ 0 &1 \\ 1 & 0  \end{pmatrix}\!,\,
  n_2=\begin{pmatrix} 0 & 0 \\ 0 & 0 \\ 1 & 0 \\ 1 & 1   \end{pmatrix}\!,\,
  n_3=\begin{pmatrix} 0 & 1 \\ 0 & 0  \\ 1 & 1\\ 0 & 1 \end{pmatrix}\!.
$$
The vectors $n_1,n_2,n_3$ generate the subgroup $N$ in $ (\mathbb{Z}/2\mathbb{Z})^8$.
One checks that the triple $\{m_1,m_2,m_3\}$ is azygetic
and that the three cosets $m_i + N$ consist of even elements only.
The computations to be described next were done with the {\tt Sage} library
{\tt abelfunctions} \cite{abelfunctions}. 

The algorithm in \cite{DvH} finds the Riemann matrix $\tau \in \mathfrak{J}_g$ of a 
plane curve in  $\mathbb{C}^2$. It is implemented in {\tt abelfunctions}.  We first check
that (\ref{eq:schottkyigusa})  does indeed vanish for such~$\tau$.

\begin{example}\label{exp:schottkycurve} \rm
The plane curve $\,y^5+x^3-1 = 0\,$ has genus four. Its 
  Riemann matrix $\tau$~is
 \[
  \begin{pmatrix}   \,\,0.16913 + 1.41714\textnormal{i}  &  -0.81736 - 0.25138\textnormal{i} & \! -0.05626 - 0.44830\textnormal{i}  & \,\,0.24724 + 0.36327\textnormal{i}\\
\! -0.81736 - 0.25138\textnormal{i}  & -0.31319 + 0.67096\textnormal{i} &\! -0.02813 - 0.57155\textnormal{i}  & \,\,0.34132 + 0.40334\textnormal{i}\\
\!-0.05626 - 0.44830\textnormal{i} & -0.02813 - 0.57155\textnormal{i}  &  \,0.32393 + 1.44947\textnormal{i}  & -0.96494 - 0.63753\textnormal{i} \\
  \, \, 0.24724 + 0.36327\textnormal{i} &  \,\,\, 0.34132 + 0.40334\textnormal{i} & \! \! -0.96494 - 0.63753\textnormal{i}   & \,\,0.62362 + 0.73694\textnormal{i} \end{pmatrix} \! .
 \]
Evaluating the $16$ theta constants $\theta[m](\tau,0)$ numerically with {\tt abelfunctions}, we find that
\begin{eqnarray*}
 \pi_1^2+\pi_2^2+\pi_3^2&=&-5.13472888270289 + 6.13887870578982\textnormal{i},  \\
 2(\pi_1\pi_2+\pi_1\pi_3+\pi_2\pi_3)&=&-5.13472882638710 + 6.13887931435788\textnormal{i}.
\end{eqnarray*}
We trust that (\ref{eq:schottkyigusa}) is zero, and conclude that
$\tau$ lies in the Schottky locus $\mathfrak{J}_4$, as expected.
\end{example}

Suppose now that we are given a matrix  $\tau$ that depends on one or two parameters,
so it traces out a curve or surface in $\mathfrak{H}_4$. Then we can use our numerical
method to determine the Schottky locus  inside that curve or surface. Here is an illustration
for a surface in $\mathfrak{H}_4$.

 \begin{example}\label{exp:shimuracurve} \rm
 The following one-parameter family of genus $4$ curves is found in \cite[\S 2]{GM}:
 $$ y^6 \,\,=\,\, x(x+1)(x-t). $$
 This is both a Shimura curve and a Teichm\"uller curve.
Its Riemann matrix is  $\rho(t) = Z_2^{-1} Z_1$ where
$Z_1,Z_2$ are given in \cite[Prop.~6]{GM}.
Consider the following two-parameter family  in $\mathfrak{H}_4$:
\begin{equation}
\label{eq:taufamily}
 \tau(s,t)  \,= \, s \cdot {\rm diag}(2,3,5,7) \,\,+\,\, \rho(t).
\end{equation}
We are interested in the restriction of
the Schottky locus $\mathfrak{J}_4$ to the $(s,t)$-plane.
For our experiment, we assume that the two parameters satisfy
 $s \in [-0.5,0.5]$ and $ \lambda^{-1}(t) \in [\,\textnormal{i},\textnormal{i}+1]$,
 where $\lambda$ is the function  in \cite[Prop.~6]{GM}. 
Using {\tt abelfunctions}, we computed the absolute value of
the modular form (\ref{eq:schottkyigusa})
at 6400 equally spaced rational points in the square
$[-0.5,0.5] \times [i,i+1]$. That graph 
 is shown in Figure \ref{eq:taufamily}.
For $s$ different from zero, the smallest absolute value of
(\ref{eq:schottkyigusa}) is $4.3 \times 10^{-3}$.
For $s=0$, all absolute values are below $ 2.9 \times 10^{-8}$.
Based on this numerical evidence, we conclude that the
Schottky locus of our family is the line $s=0$.
\end{example}

\begin{figure}[h]
  \begin{center}
\includegraphics[height=8.4cm]{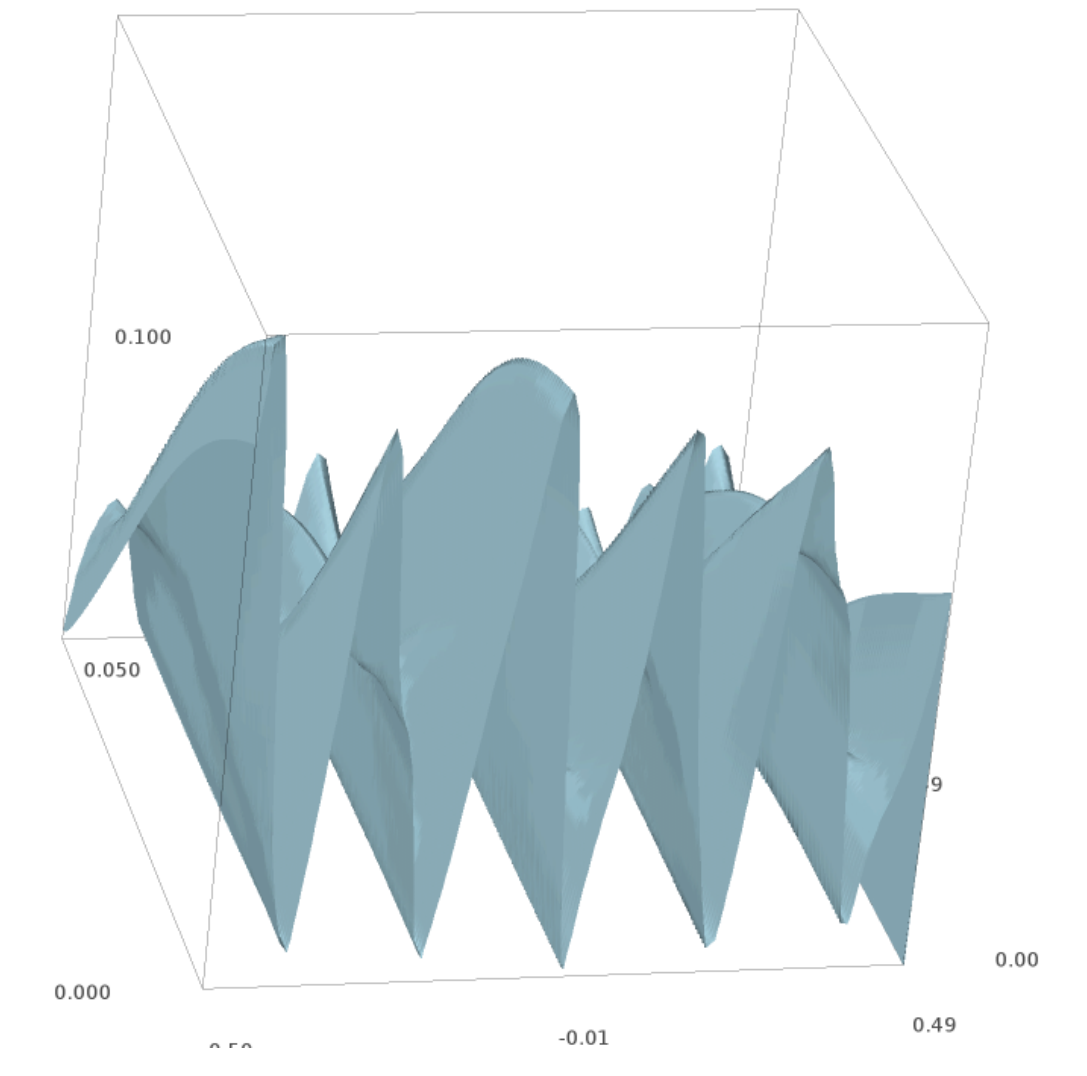}
\vspace{-0.2in}
\end{center}    \caption{\label{fig:eins}
Absolute value of the Schottky--Igusa modular form
on the 2-parameter family~(\ref{eq:taufamily}).
}
 \end{figure}

We now come to the Schottky Recovery Problem. Our input is a matrix
$\tau$ in $ \mathfrak{J}_4$. Our task is to compute a curve whose
Riemann matrix equals $\tau$. We use the following result
from Kempf's paper \cite{kempf}.
The {\em theta divisor} in the Jacobian 
$\, \mathbb{C}^4/(\mathbb{Z}^4 + \mathbb{Z}^4 \tau )\,$
is the zero locus $\Theta^{-1}(0)$ of the Riemann theta function
$\,\Theta(z) := \theta[0](\tau,z)$. For generic $\tau$ this divisor is singular at
precisely two points. These represent $3$-to-$1$ maps from
the curve to $\PP^1$. We compute a vector $z^* \in \mathbb{C}^4$ that is
a singular point of $\Theta^{-1}(0)$ by solving the system of five equations
\begin{equation}
\label{eq:findsingular} 
\Theta(z) \, = \, \frac{\partial \Theta}{\partial z_1}(z) = 
 \frac{\partial \Theta}{\partial z_2}(z) = 
  \frac{\partial \Theta}{\partial z_3}(z) = 
   \frac{\partial \Theta}{\partial z_4}(z) =  0.
\end{equation}
The Taylor series of the Riemann theta function
$\Theta$ at the singular point $z^*$ has the form
\begin{equation}
\label{eq:taylorseries}
\Theta(z^*+x) \,= \, f_2(x) \,+\, f_3(x) \, + \,f_4(x) \,+ \,\,\hbox{higher order terms}, 
\end{equation}
where $f_s$ is a homogeneous polynomial of degree $s$ in $ x = (x_1,x_2,x_3,x_4)$.

\begin{proposition}[Kempf \cite{kempf}]
The canonical curve with Riemann matrix $\tau$ is the
degree $6$ curve in $\PP^3$ that is defined by the quadratic equation $f_2 = 0$
and the cubic equation $ f_3 = 0$.
\end{proposition}

Thus our algorithm for the Schottky Recovery Problem consists of solving the 
five equations (\ref{eq:findsingular}) for $z^* \in \mathbb{C}^4 $, followed by
extracting the polynomials $f_2$ and $f_3$ in the Taylor series (\ref{eq:taylorseries}). Both of 
these steps can be done numerically using the software {\tt abelfunctions} \cite{abelfunctions}.

\begin{example}  \rm
Let $\tau \in \mathfrak{J}_4$ be the Riemann matrix of the genus $4$ curve $C = \{x^3 y^3 +x^3 + y^3 = 1\}$. 
We obtain $\tau$ numerically using {\tt abelfunctions}. We want to recover $C$ from $\tau$. To be precise,
given only $\tau$,  we want to find defining equations
$f_2 = f_3 = 0$ in $\PP^3$ of the canonical embedding of $C$. For that we use evaluations of $\Theta(z)$ and its derivatives in
{\tt abelfunctions}, combined with a numerical optimization routine in {\tt SciPy} \cite{spicy}.
We solve the equations (\ref{eq:findsingular}) starting from random
points $z = u + \tau v$ where $u,v \in \RR^4$
with entries between $0$ and $1$. After several tries,
the local method in {\tt SciPy} converges to the following solution of our equations:
$$ z^* = \bigl(0.55517+0.69801\textnormal{i},  0.53678+0.26881\textnormal{i}, -0.50000-0.58958\textnormal{i},  0.55517+0.69801\textnormal{i}\bigr) . $$
Using (\ref{eq:taylorseries}), we computed the quadric $f_2$, which is nonsingular, as well as the cubic $f_3$:
$$
\begin{small}
 \begin{matrix}
  f_2(x) & =&  
(-3.044822827 + 21.980542613 \textnormal{i})\cdot x_1^2   +( - 237.95207224  + 252.54744634 \textnormal{i})\cdot x_1 x_2 \\   & & 
\,+(- 222.35552015  + 139.95612952 \textnormal{i})\cdot x_1 x_3   +(- 200.66932133  - 16.596272620 \textnormal{i})\cdot x_1 x_4 \\  & &
  + (- 191.16241727 - 85.22650070 \textnormal{i})\cdot x_2^2   + (- 429.11449060 + 167.32094535 \textnormal{i})\cdot x_2 x_3 \\  & & 
+(  - 237.952072 + 252.54744632 \textnormal{i})\cdot x_2 x_4   +(- 206.75896934 + 27.364814282 \textnormal{i})\cdot x_3^2 \\ & & 
+(  222.35552013 + 139.95612953 \textnormal{i})\cdot x_3 x_4   + (- 3.0448227745 + 21.980542601 \textnormal{i})\cdot x_4^2
 \end{matrix}
  \end{small} $$ $$ \begin{small}
\begin{matrix}
f_3(x) \quad =  \quad
(441.375966486 + 61.14097461986 \textnormal{i})\cdot x_1^3   +(  2785.727151434 + 2303.609067429 \textnormal{i})\cdot x_1^2 x_2 \\
+ \quad \cdots\,\,\cdots \quad   + \quad ( 441.3759668263 + 61.14097402189 \textnormal{i})\cdot x_4^3.
\end{matrix}
\end{small}
$$
As a proof of concept we also computed the $120$ tritangent planes numerically directly from $\tau$.
    These planes are indexed by the $120$ odd theta characteristics $m$.
    In analogy to the computation  in \cite[\S 5.2]{SD} of the $28$ bitangents for $g=3$,
    their defining equations are
$$ 
\frac{\partial \theta[m](\tau,z)}{\partial z_1}\biggl|_{z=0} \!\!\! \cdot \,x_1 \,\, + \,\,
\frac{\partial \theta[m](\tau,z)}{\partial z_2}\biggl|_{z=0} \!\!\! \cdot \,x_2 \,\, + \,\,
\frac{\partial \theta[m](\tau,z)}{\partial z_3}\biggl|_{z=0} \!\!\! \cdot \,x_3 \,\, + \,\,
\frac{\partial \theta[m](\tau,z)}{\partial z_4}\biggl|_{z=0} \!\!\! \cdot \,x_4 \,\,\, = \,\,\, 0. 
$$
We verified numerically that each such plane
meets $\{f_2 = f_3 = 0\}$ in three double points.
\end{example}

\begin{remark} \rm
On our website (\ref{eq:url}), we offer a program in {\tt Sage} whose input is
a symmetric $4 \times 4$-matrix $\tau \in \mathfrak{H}_4$, given numerically.
The code decides whether $\tau $ lies in $\mathfrak{J}_4$ and, in the affirmative case, it computes
the canonical curve $\{f_2 = f_3 = 0\}$ and its 120 tritangent planes.
\end{remark}

\section{The Tropical Schottky Problem}\label{sec:tropical}

Curves, their Jacobians, and the Schottky locus have
natural counterparts in the combinatorial setting of tropical geometry.
We review the basics from \cite{BBC, BMV, chan12, mz}.
The role of a curve is played by a connected metric graph $\Gamma = (V,E,l,w)$.
This has vertex set $V$, edge set $E$, a length function $l: E \rightarrow \mathbb{R}_{> 0}$,
and a weight function $w : V \rightarrow \mathbb{Z}_{\geq 0}$.
The genus of $\Gamma$ is
\begin{equation}
\label{eq:graphgenus}
g \,\,= \,\, |E| -|V|  +1 \,+\, \sum_{v\in V} w(v). 
\end{equation}
The moduli space $\mathcal{M}_g^{\rm trop}$ comprises all metric graphs of genus $g$.
This is a stacky fan of dimension $3g-3$.
See \cite[Figure 4]{chan12} for a colorful illustration.
The tropical Torelli map $\,\mathcal{M}_g^{\rm trop} \rightarrow \mathfrak{H}_g^{\rm trop}\,$
takes $\Gamma$ to its (symmetric and positive semidefinite)  Riemann matrix $Q_\Gamma$.
 
Fix a basis for the integral homology $H_1(\Gamma,\mathbb{Z}) \simeq \mathbb{Z}^g$. 
Beside the usual cycles in $\Gamma$, this group has $w(v)$
generators for the virtual cycles at each vertex $v$.
Let $B$ denote the $g \times |E|$ matrix whose columns
record the coefficients of each edge in the basis vectors.
Let $D$ be the $|E| \times |E|$ diagonal matrix whose entries are the edge lengths.
The {\em Riemann matrix} of $\Gamma$~is
\begin{equation}
\label{eq:tropicalriemann} Q_\Gamma \,\, = \,\, B \cdot D \cdot B^t.
\end{equation}
One way to choose a basis is to fix an orientation 
and a spanning tree of $\Gamma$.
Each edge not in that tree then determines a cycle
with $\pm 1$-coefficients.
See \cite[\S 4]{BBC} for details and an example.
Changing the basis of $H_1(\Gamma,\mathbb{Z})$
corresponds to the action of ${\rm GL}_g(\mathbb{Z})$ on
$Q_\Gamma$ by conjugation.

The matrix $Q_\Gamma$ has rank $g - \sum_{v\in V} w(v)$.
We defined $\mathfrak{H}_g^{\rm trop}$ with positive definite matrices.
Those have rank $g$.
For that reason, we now
restrict to graphs with
zero weights, i.e.~$w \equiv 0$.

The {\em tropical Schottky locus}  $\mathfrak{J}_g^{\rm trop}$ 
is the set of all matrices (\ref{eq:tropicalriemann}),
where $\Gamma = (V,E,l)$ runs over graphs
of genus $g$, and $B$ runs over their cycle bases.
This set is known as the {\em cographic locus}
in $\mathfrak{H}_g^{\rm trop}$, because
the $g \times |E|$ matrix $B$ is a representation
of the {\em cographic matroid} of~$\Gamma$.

The Schottky Decision Problem asks for a test of
membership in $\mathfrak{J}_g^{\rm trop}$. To be precise,
given a positive definite matrix $Q$, does there exist
a metric graph $\Gamma$ such that $Q = Q_\Gamma$?

To address this question, we need the polyhedral fan structures
on $\mathfrak{J}_g^{\rm trop}$ and  $\mathfrak{H}_g^{\rm trop}$.
Let $G = (V,E)$ be the graph underlying $\Gamma$,
with $E = \{e_1,e_2,\ldots,e_m\}$. Fix a cycle basis as above.
Let $b_1,b_2,\ldots,b_m$ be the column vectors of the $g \times m$-matrix $B$.
Formula (\ref{eq:tropicalriemann}) is equivalent~to
\begin{equation}
\label{eq:tropicalriemann2}
Q_\Gamma \,\,=\,\, l(e_1) b_1b_1^t + l(e_2) b_2 b_2^t +  \cdots +  l(e_m)b_m b_m^t.
 \end{equation}
The cone of all Riemann matrices for the graph $G$, allowing the edge lengths to vary, is 
\begin{equation}  \label{eqn-cone}
\sigma_{G,B} \,\,=\,\, \mathbb{R}_{>0} \bigl\{ b_1b_1^t, \,b_2 b_2^t, \,\ldots\,,\, b_m b_m^t\bigr\}\,.
\end{equation}
This is a relatively open rational convex polyhedral cone,
 spanned by matrices of rank $1$.
The collection of all cones $\sigma_{G,B}$ is a polyhedral fan whose support
is the Schottky locus $\mathfrak{J}_g^{\rm trop}$.

This fan is a subfan of the
{\em second Voronoi decomposition} of the cone $\mathfrak{H}_g^{\rm trop}$ of positive
definite matrices.
 The latter fan is defined as follows. Fix
a Riemann matrix $Q \in \mathfrak{H}_g^{\rm trop}$ and
consider its quadratic form $\,\mathbb{Z}^g \rightarrow \mathbb{R},\,
x \mapsto x^t Q x$. The values of this quadratic form
define a regular polyhedral subdivision of $\mathbb{R}^g$
with vertices at $\mathbb{Z}^g$. This is denoted ${\rm Del}(Q)$ and
known as the {\em Delaunay subdivision} of $Q$.
Dual to ${\rm Del}(Q)$ is the {\em Voronoi decomposition} of
$\mathbb{R}^g$. The cells of the Voronoi decomposition of $Q$ are
the lattice translates of the \emph{Voronoi polytope}
\begin{equation}
\label{eq:voronoipolytope}
 \bigl\{\,  p\in\mathbb{R}^g \,:\, 2 p^t Q x \leq x^t Q x
 \,\,\,\hbox{for all} \,\,x\in\mathbb{Z}^g \,\bigr\}.
 \end{equation}
This is the set of points in $\mathbb{R}^g$ for which the origin is the closest lattice point,
in the norm given by $Q$.
If $Q$ is generic then
the Delaunay subdivision is a triangulation
and the Voronoi polytope (\ref{eq:voronoipolytope}) is simple.
It is dual to the link of the origin in the simplicial complex ${\rm Del}(Q)$.

The structures above represent principally polarized abelian varieties
in tropical geometry. A tropical abelian variety is the torus
$\mathbb{R}^g/\mathbb{Z}^g$ together with a quadratic form $Q \in \mathfrak{H}_g^{\rm trop}$.
The {\em tropical theta divisor} is given by the  codimension one cells in the 
induced Voronoi decomposition of $\mathbb{R}^g/\mathbb{Z}^g$.
See \cite[\S 5]{BBC} for an introduction with many pictures and many references.

We now fix an arbitrary Delaunay subdivision $D$ of $\mathbb{R}^g$.
Its {\em secondary cone} is defined as
\begin{align} \label{def-sec-cone}
\sigma_D \,\,=\,\, \bigl\{ \,Q\in \mathfrak{H}_g^{\rm trop}
\,\,|\,\, \Del(Q) = D \bigr\}\,.
\end{align}
This is a relatively open convex polyhedral cone. It consists of
positive definite matrices $Q$ whose Voronoi polytopes (\ref{eq:voronoipolytope})
have the same normal fan. The group ${\rm GL}_g(\mathbb{Z})$ acts on the
set of secondary cones. In his classical reduction theory for quadratic forms, 
Voronoi \cite{voronoi} proved that the cones $\sigma_D$ form a polyhedral fan,
now known as the {\em second Voronoi decomposition} of $\mathfrak{H}_g^{\rm trop}$,
and that there are only finitely many secondary cones $\sigma_D$ up to the action
of ${\rm GL}_g(\mathbb{Z})$. The following summarizes  characteristic features for matrices
in the Schottky locus $\mathfrak{J}_g^{\rm trop}$.

\begin{proposition} \label{prop:features}
Fix a graph $G$ with metric $D$,
homology basis $B$, and Riemann matrix  $Q = BDB^t$.
The Voronoi polytope (\ref{eq:voronoipolytope}) is affinely isomorphic to the
zonotope $\sum_{i=1}^m [-b_i,b_i]$. 
The secondary cone $\sigma_{{\rm Del}(Q)}$ is spanned
 by the rank one matrices $b_i b_i^t$: it equals $\sigma_{G,B}$ in
(\ref{eqn-cone}).
\end{proposition}

\begin{proof}
This can be extracted from Vallentin's thesis \cite{frank}.
The affine isomorphism is given by the invertible matrix
$Q$, as explained in item iii) of \cite[\S 3.3.1]{frank}.
The Voronoi polytope being the zonotope 
 $\sum_{i=1}^m [-b_i,b_i]$ follows from
the discussion on cographic lattices in \cite[\S 3.5]{frank}.
The result for the secondary cone is derived from \cite[\S 2.6]{frank}.
See \cite[\S 4]{frank} for many examples.
\end{proof}

We now fix $g=4$. Vallentin \cite[\S 4.4.6]{frank} lists  all $52$ combinatorial types of Delaunay subdivisions 
of $\mathbb{Z}^4$. His table contains the f-vectors of all $52$ Voronoi polytopes. Precisely
$16$ of these types are cographic, and these comprise the Schottky locus $\mathfrak{J}_4^{\rm trop}$.
These are described in rows 3 to 18 of the table in \cite[\S 4.4.6]{frank}. We reproduce 
the relevant data in Table \ref{table-trop-schottky}. The following key lemma is found by
inspecting Vallentin's list of f-vectors.

\begin{lemma} The f-vectors of the $16$ Voronoi polytopes representing the Schottky locus
$\mathfrak{J}_4^{\rm trop}$ are distinct from the f-vectors of the other $36$ Voronoi polytopes,
corresponding to $\mathfrak{H}_4^{\rm trop} \backslash \mathfrak{J}_4^{\rm trop}$.
\end{lemma}

This lemma gives rise to the following method for the tropical Schottky decision  problem.

\begin{algorithm}[Tropical Schottky Decision]   \label{alg:TSD}
\underbar{Input}: $Q \in \mathfrak{H}_4^{\rm trop}\!$.
\underbar{Output}: Yes, if $\,Q \in \mathfrak{J}_4^{\rm trop}\!$. \\
1. Compute the Voronoi polytope in (\ref{eq:voronoipolytope}) for the quadratic form $Q$. \\
2. Determine the f-vector $(f_0,f_1,f_2,f_3)$ of this $4$-dimensional polytope. \\
3. Check whether this f-vector appears in our Table \ref{table-trop-schottky}.
Output ``Yes'' if this holds.
\end{algorithm}

\begin{longtable}{|c|c|c|c|c|c|c|}
  \hline
Graph $G$ & Riemann matrix $Q_\Gamma$ & $f_0$ & $f_1$ &
$f_2$ & $f_3$ & Dimension of $\sigma_D$ \\
\hline \rule{0pt}{1.4\normalbaselineskip}
\begin{tikzpicture}[baseline={([yshift=-.5ex]current bounding box.center)}] \filldraw (0,0) circle (1pt) (0,1) circle (1pt)
  (0.5,0.5) circle (1pt) (1,0.5) circle (1pt) (1.5,0) circle (1pt) (1.5,1) circle (1pt); \draw (0,0)--(0,1)--(0.5,0.5)--(0,0) (0,1)--(1.5,1)--(1.5,0)--(0,0) (0.5,0.5)--(1,0.5) (1.5,0)--(1,0.5)--(1.5,1);
\end{tikzpicture}&
$\begin{psmallmatrix}3&1&-1&0\\1&4&1&1\\-1&1&4&-1\\0&1&-1&3\end{psmallmatrix}$
  & 96 & 198 & 130 & 28 & 9 
\\[4mm] \begin{tikzpicture}[baseline={([yshift=-.5ex]current bounding box.center)}]
  \filldraw (0,0) circle (1pt) (0,1) circle (1pt) (1,0) circle (1pt) (2,0) circle (1pt)
  (1,0) circle (1pt) (1,1) circle (1pt) (2,1) circle (1pt);\draw
  (0,0)--(0,1) (0,0)--(1,1) (0,0)--(2,1) (1,0)--(0,1) (1,0)--(1,1) (1,0)--(2,1) (2,0)--(0,1) (2,0)--(1,1) (2,0)--(2,1);
\end{tikzpicture}& $\begin{psmallmatrix}4&2&-2&-1\\2&4&-1&-2\\-2&-1&4&2\\-1&-2&2&4\end{psmallmatrix}$
  & 102 & 216 & 144 & 30 & 9 \\[4mm] \hline \rule{0pt}{1.4\normalbaselineskip}
\begin{tikzpicture}[baseline={([yshift=-.5ex]current bounding box.center)}]
\filldraw (0,0) circle (1pt) (0,1) circle (1pt) (1,0) circle (1pt)
(1,1) circle (1pt) (0.5,0.5) circle (1pt);
\draw (0,0)--(1,0) (1,1)--(0,1)--(0,0) (0,0)--(0.5,0.5)--(1,1)
(0.5,0.5)--(0,1) (1,0) to [bend left] (1,1) (1,0) to [bend right] (1,1);
\end{tikzpicture} 
& $\begin{psmallmatrix}2&0&-1&0\\0&3&1&1\\-1&1&4&-1\\0&1&-1&3\end{psmallmatrix}$ & 72 & 150 & 102 & 24 & 8 \\[4mm]
\begin{tikzpicture}[baseline={([yshift=-.5ex]current bounding box.center)}]
\filldraw (0,0) circle (1pt) (0,1) circle (1pt) (1,0) circle (1pt)
(1,1) circle (1pt) (0.5,0.5) circle (1pt);
\draw (0,0)--(1,0)--(1,1)--(0,1)--(0,0) (0,0)--(0.5,0.5)--(1,1) (1,0)--(0.5,0.5)--(0,1);
\end{tikzpicture} 
& $\begin{psmallmatrix}3&2&1&-1\\2&4&2&-1\\1&2&4&1\\-1&-1&1&3\end{psmallmatrix}$ & 78 & 168 & 116 & 26 & 8 \\[4mm] \hline\rule{0pt}{1.4\normalbaselineskip}
\begin{tikzpicture}[baseline={([yshift=-.5ex]current bounding box.center)}]
\filldraw (0,0) circle (1pt) (0,1) circle (1pt) (1,0) circle (1pt)
(1,1) circle (1pt);
\draw (0,0)--(1,0) (0,0)--(0,1) (0,0)--(1,1) (1,0) to [bend left]
(1,1) (1,0) to [bend right] (1,1) (0,1)--(1,1) (0,1)--(1,0);
\end{tikzpicture}
& $\begin{psmallmatrix}3&1&-1&-1\\1&3&1&1\\-1&1&3&2\\-1&1&2&3\end{psmallmatrix}$ & 60 & 134 & 98 & 24 & 7 \\[4mm]
 \begin{tikzpicture}[baseline={([yshift=-.5ex]current bounding box.center)}]
\filldraw (0,0) circle (1pt) (0,1) circle (1pt) (1,0) circle (1pt)
(1,1) circle (1pt);
\draw (0,0)--(1,0) (1,0) to [bend left] (1,1) (1,0) to [bend right]
(1,1) (0,1) to [bend left] (1,1) (0,1) to [bend right] (1,1)
(0,0) to [bend left] (0,1) (0,0) to [bend right] (0,1); 
\end{tikzpicture} 
& $\begin{psmallmatrix}2&0&-1&-1\\0&2&-1&-1\\-1&-1&4&3\\-1&-1&3&4\end{psmallmatrix}$ &
  54 & 116 & 84 & 22 & 7 \\[4mm]
\begin{tikzpicture}[baseline={([yshift=-.5ex]current bounding box.center)}]
\filldraw (0,0) circle (1pt) (0,1) circle (1pt) (1,0) circle (1pt)
(1,1) circle (1pt);
\draw (0,0)--(1,0) (0,0)--(0,1) (0,0)--(1,1) (1,0) to [bend left]
(1,1) (1,0) to [bend right] (1,1) (0,1) to [bend left]
(1,1) (0,1) to [bend right] (1,1);
\end{tikzpicture} 
& $\begin{psmallmatrix}2&0&-1&0\\0&2&0&-1\\-1&0&3&1\\0&-1&1&3\end{psmallmatrix}$
 & 54 & 114 & 80 & 20 & 7 \\[4mm]
\begin{tikzpicture}[baseline={([yshift=1ex]current bounding box.center)}]
\filldraw (0,0) circle (1pt) (0,0.8) circle (1pt) (0.8,0) circle (1pt)
(0.8,0.8) circle (1pt);
\draw (0,0)--(0.8,0) (0,0)--(0,0.8) (0,0)--(0.8,0.8) (0.8,0) (0.8,0)--(0.8,0.8)
(0,0.8)--(0.8,0.8) (0,0.8)--(0.8,0)
(0.8,0) to [in=30,out=-60,distance=8mm] (0.8,0);
\end{tikzpicture} & $\begin{psmallmatrix}3&1&-1&0\\1&3&1&0\\-1&1&3&0\\0&0&0&1\end{psmallmatrix}$ & 48 & 96 & 64 & 16 & 7 \\[-3mm] \hline \rule{0pt}{1.4\normalbaselineskip}
\begin{tikzpicture}[baseline={([yshift=-.5ex]current bounding box.center)}]
\filldraw (0,1) circle (1pt) (1,1) circle (1pt) (0.5,0) circle (1pt);
\draw (0.5,0) to [bend left] (0,1) (0.5,0) to [bend right] (0,1)
(0.5,0) to [bend left] (1,1) (0.5,0) to [bend right] (1,1) 
(0,1) to [bend left] (1,1) (0,1) to [bend right] (1,1);
\end{tikzpicture} 
& $\begin{psmallmatrix}2&0&-1&-1\\0&2&1&1\\-1&1&3&2\\-1&1&2&3\end{psmallmatrix}$ & 46 & 108 & 84 & 22 & 6 \\[4mm]
\begin{tikzpicture}[baseline={([yshift=-.5ex]current bounding box.center)}]
\filldraw (0,1) circle (1pt) (1,1) circle (1pt) (0.5,0) circle (1pt);
\draw (0.5,0) to [bend left] (0,1) (0.5,0) to [bend right] (0,1)
(0.5,0) to [bend left] (1,1) (0.5,0) to [bend right] (1,1) 
(0,1)--(1,1) (0.5,0)--(1,1);
\end{tikzpicture}
& $\begin{psmallmatrix}2&-1&-1&-1\\-1&3&2&2\\-1&2&3&2\\-1&2&2&3\end{psmallmatrix}$ & 42 & 94 & 72 & 20 & 6 \\[4mm]
\begin{tikzpicture}[baseline={([yshift=-.5ex]current bounding box.center)}]
\filldraw (0,0.6) circle (1pt) (0.6,0.6) circle (1pt) (0.3,0) circle (1pt);
\draw (0.3,0) to [bend left] (0,0.6) (0.3,0) to [bend right] (0,0.6)
(0.3,0) to [bend left] (0.6,0.6) (0.3,0) to [bend right] (0.6,0.6) 
(0,0.6)--(0.6,0.6) (0.3,0) to [in=-30,out=-150, distance=8mm] (0.3,0);
\end{tikzpicture}
& $\begin{psmallmatrix}2&1&1&0\\1&3&2&0\\1&2&3&0\\0&0&0&1\end{psmallmatrix}$ & 36 & 74 & 52 & 14 & 6 \\[4mm]
\begin{tikzpicture}[baseline={([yshift=-.5ex]current bounding box.center)}]
\filldraw (0,0) circle (1pt) (0.8,0) circle (1pt) (1.6,0) circle (1pt);
\draw (0,0) to [bend left=40] (0.8,0) (0,0) to [bend right=40] (0.8,0) (0,0)--(0.8,0);
\draw (0.8,0) to [bend left=40] (1.6,0) (0.8,0) to [bend right=40] (1.6,0) (0.8,0)--(1.6,0);
\end{tikzpicture}
& $\begin{psmallmatrix}2&1&0&0\\1&2&0&0\\0&0&2&1\\0&0&1&2\end{psmallmatrix}$ & 36 & 72 & 48 & 12 & 6 \\[4mm] \hline \rule{0pt}{1.4\normalbaselineskip}
\begin{tikzpicture}[baseline={([yshift=-.5ex]current bounding box.center)}]
\filldraw (0,0) circle (1pt) (1,0) circle (1pt);
\draw (0,0) to [bend left=20] (1,0) (0,0) to [bend left=60] (1,0)
(0,0) to [bend right=20] (1,0) (0,0) to [bend right=60] (1,0) (0,0)--(1,0);
\end{tikzpicture} 
& $\begin{psmallmatrix}2&1&1&1\\1&2&1&1\\1&1&2&1\\1&1&1&2\end{psmallmatrix}$ & 30 & 70 & 60 & 20 & 5 \\[3mm]
\begin{tikzpicture}[baseline={([yshift=-.5ex]current bounding box.center)}]
\filldraw (0,0) circle (1pt) (1,0) circle (1pt);
\draw (0,0) to [bend left=20] (1,0) (0,0) to [bend left=60] (1,0)
(0,0) to [bend right=20] (1,0) (0,0) to [bend right=60] (1,0);
\draw (1,0) to[in=50,out=-50, distance=8mm,looseness=50] (1,0);
\end{tikzpicture} 
& $\begin{psmallmatrix}2&1&1&0\\1&2&1&0\\1&1&2&0\\0&0&0&1\end{psmallmatrix}$ & 28 & 62 & 48 & 14 & 5 \\[-2mm]
\begin{tikzpicture}[baseline={([yshift=-.5ex]current bounding box.center)}]
\filldraw (0.2,0) circle (1pt) (1,0) circle (1pt);
\draw (0.2,0) to [bend left=40] (1,0) (0.2,0) to [bend right=40] (1,0) (0.2,0)--(1,0);
\draw (1,0) to[in=90,out=0, distance=8mm] (1,0) (1,0) to[in=0,out=-90, distance=8mm] (1,0);
\end{tikzpicture} 
 & $\begin{psmallmatrix}2&1&0&0\\1&2&0&0\\0&0&1&0\\0&0&0&1\end{psmallmatrix}$   & 24 & 48 & 34 & 10 & 5 \\[-2mm] \hline 
\begin{tikzpicture}[baseline={([yshift=-.5ex]current bounding box.center)}]
\filldraw (0,0) circle (1pt);
\draw (0,0) to[in=0, out=90,distance=8mm] (0,0) (0,0) to[in=90, out=180,distance=8mm](0,0) (0,0) to[in=180, out=270,distance=8mm] (0,0) (0,0) to[in=270, out=0,distance=8mm] (0,0);
\end{tikzpicture} 
& $\begin{psmallmatrix}1&0&0&0\\0&1&0&0\\0&0&1&0\\0&0&0&1\end{psmallmatrix}$ & 16 & 32 & 24 & 8 & 4 \\ \hline 
\caption{The tropical Schottky locus for $g=4$}
\label{table-trop-schottky}
\end{longtable}

We implemented Algorithm \ref{alg:TSD} using existing software 
for polyhedral geometry, namely the \texttt{GAP} package \texttt{polyhedral} 
due to Dutour Sikiri\'c \cite{sikiric, sgsw}, as well as Joswig's 
 {\tt polymake} \cite{polymake}.

The first column of Table \ref{table-trop-schottky} shows all relevant graphs $G$ of genus $4$.
The second column gives a representative Riemann matrix.
Here all edges have length $1$ and a cycle basis $B$ was chosen.
Using (\ref{eqn-cone}), we also  precomputed 
the secondary cones  $\sigma_{G,B}$ for the $16$
 representatives.

\begin{example}\label{ex-decision} \rm
 Using the \texttt{GAP} package \texttt{polyhedral} \cite{sikiric} we compute the Voronoi polytope of
$$ \begin{small}
 Q\,\,=\,\,\begin{pmatrix}14&-9&11&0\\-9&11&-2&1\\11&-2&21&11\\0&1&11&14\end{pmatrix}.
 \end{small}
 $$
Its $f$-vector is $(62,142,104,24)$. This does not appear in
Table \ref{table-trop-schottky}.
Hence $Q$ is not in  $\mathfrak{J}_4^{\rm trop}$.
\end{example}

We now address the Schottky Recovery Problem.
The input is a matrix $Q \in \mathfrak{J}_4^{\rm trop}$.
From Algorithm \ref{alg:TSD} we know the f-vector
of the Voronoi polytope. Using Table \ref{table-trop-schottky},
this uniquely identifies the graph $G$.
Note that our graphs $G$ are dual to those in \cite[\S 4.4.4]{frank}.
From our precomputed list, we also know the secondary cone $\sigma_{G,B}$
for some choice of basis~$B$.

\begin{algorithm}[Tropical Schottky Recovery]   \label{alg:TSR}
\underbar{Input}: $Q \in \mathfrak{J}_4^{\rm trop}$. \\
\underbar{Output}: A metric graph $\Gamma $ whose Riemann matrix $Q_\Gamma$ equals $Q$. \\
1. Identify the underlying graph $G$  from Table \ref{table-trop-schottky}.
Retrieve the basis $B$ and the cone $\sigma_{G,B}$. \\
2. Let $D = {\rm Del}(Q)$ and compute the secondary cone $\sigma_D$ 
as in (\ref{def-sec-cone}). \\
3. The cones $\sigma_D$ and $\sigma_{G,B}$ are related by a linear transformation
$X \in {\rm GL}_4(\mathbb{Z})$.~Compute $X$. \\
4. The matrix $ X^t Q X$ lies in $\sigma_{G,B}$. Compute
$\ell_1,\ldots,\ell_m$ such that
$X^t Q X = \sum_{i=1}^m \ell_i b_i b_i^t$. \\
5. Output the graph $G$ with length $\ell_i$ for its $i$-th edge,
corresponding to the column $b_i$ of $B$.
\end{algorithm}

We implemented this algorithm as follows.
Step 2  can be done using \texttt{polyhedral} \cite{sikiric}.
This code computes the secondary cone  $\sigma_D$ containing
a given positive definite matrix $Q$.
The matrix $X \in {\rm GL}_4(\mathbb{Z}) $ in Step 3 is also
found by \texttt{polyhedral}, but with external calls to the package \texttt{isom} 
due to Plesken and Souvignier \cite{ps95}.
We refer to \cite[\S 4]{sgsw} for details.
For Step 4 we note that the rank $1$ matrices $b_1 b_1^t, \ldots, b_m b_m^t$ are linearly
independent \cite[\S 4.4.4]{frank}. Indeed, the two $9$-dimensional secondary cones $\sigma_{G,B}$ at the top of 
Table \ref{table-trop-schottky} are simplicial, and so are their faces.
Hence the multipliers $\ell_1,\ldots,\ell_m$ found in Step 4 are unique and positive.
These $\ell_i$ must agree with the desired edge lengths $l(e_i)$, by the formula
for $Q= Q_\Gamma$ in (\ref{eq:tropicalriemann2}).

\begin{example}\label{ex-recov} \rm 
Consider the Schottky Recovery Problem for the matrix
\begin{equation}
\label{eq:ex36matrix}
\begin{small}
Q \,\,=\,\,\begin{pmatrix}
17 & 5 & 3 & 5\\
5 & 19 & 7 & 11\\
3 & 7 & 23 & 16 \\
5 & 11 & 16 & 29
\end{pmatrix}.
 \end{small}
\end{equation}
Using \texttt{polyhedral}, we find that the f-vector of its
Voronoi polytope is $(96,198,130,28)$. This matches the first row in
Table \ref{table-trop-schottky}. Hence $Q \in \mathfrak{J}_4^{\rm trop}$,
and $G$ is the triangular prism.
Using \texttt{polyhedral} and \texttt{isom}, we find a matrix that
 maps $Q$ into our preprocessed secondary cone:
$$
\begin{small}
X = \begin{pmatrix}
  0 & 0 & 0 & 1 \\
  1 & 0 & 0 & 0 \\
  0 & 1 & 1 & 0 \\
  -1 & -1 & 0 & 0
\end{pmatrix}\,\in\, {\rm GL}_4(\mathbb{Z}) \qquad \hbox{gives} \qquad
Q' = X^tQX = \begin{pmatrix}
  26 & 9 & -9 & 0 \\
  9 & 20 & 7 & -2 \\
  -9 & 7 & 23 & 3 \\
  0 & -2 & 3 & 17
\end{pmatrix}\, \in\, \sigma_{G,B}.
\end{small}
$$
This $Q'$ is the Riemann matrix of the metric graph in Figure
\ref{fig-recover-trop-schottky}, with basis cycles
 $e_2+e_6-e_3$, $-e_1-e_4+e_7+e_2$, $-e_1-e_5+e_8+e_3$, and
$e_4+e_9-e_5$. These are the rows of the $4 \times 9$-matrix $B$.
In Step 4 of Algorithm \ref{alg:TSR} we compute
$D = {\rm diag}(\ell_1,\ldots,\ell_9) = {\rm diag}(
7,9,9,2,3,8,2,4,12)$. 
 In Step 5 we output the metric graph in 
Figure \ref{fig-recover-trop-schottky}.
Its Riemann matrix equals $Q = BDB^t$.

\begin{figure}[h]\label{fig-graph1}
\begin{center}
\begin{tikzpicture}[scale=1.5]
\usetikzlibrary{decorations.markings}  
\begin{scope}[thick, decoration={markings, mark=at position 0.5 with {\arrow{latex}}}]
\filldraw (0,0) circle (1pt) (0,2) circle (1pt)
  (1,1) circle (1pt) (2,1) circle (1pt) (3,0) circle (1pt)
(3,2) circle (1pt);
\draw[postaction={decorate}] (1,1)--node[above,red]{7}node[below]{$e_1$}(2,1);
\draw[postaction={decorate}] (1,1)--node[above,red]{9}node[below]{$e_2$}(0,2);
\draw[postaction={decorate}] (1,1)--node[above,red]{9}node[below]{$e_3$}(0,0);
\draw[postaction={decorate}] (2,1)--node[above,red]{2}node[below]{ $e_4$}(3,2);
\draw[postaction={decorate}] (2,1)--node[above,red]{3}node[below]{$e_5$}(3,0);
\draw[postaction={decorate}] (0,2)--node[right,red]{8}node[left]{$e_6$}(0,0);
\draw[postaction={decorate}] (0,2)--node[above,red]{2}node[below]{$e_7$}(3,2);
\draw[postaction={decorate}] (0,0)--node[above,red]{4}node[below]{$e_8$}(3,0);
\draw[postaction={decorate}] (3,2)--node[right,red]{12}node[left]{$e_9$}(3,0);
\end{scope}
\end{tikzpicture}
\caption{Metric graph with edge lengths in red. Its
Riemann matrix matches (\ref{eq:ex36matrix}).
}
\label{fig-recover-trop-schottky}
\end{center}
\end{figure}
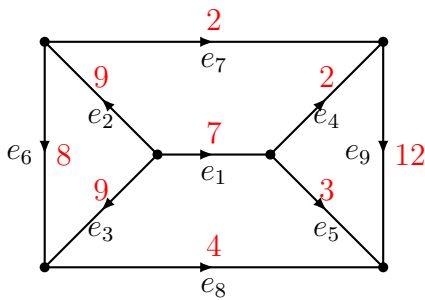
\end{example}

It is instructive to compare Algorithms \ref{alg:TSD} and
\ref{alg:TSR} with Section~\ref {sec:classic}.
Our classical solution is not just the abstract
Riemann surface but it consists of a canonical embedding into $\PP^3$.

Canonical embeddings also exist for metric graphs $\Gamma$, as
explained in \cite[\S 7]{HMY}. However, even computing the
ambient space $|K|$, that plays the role of $\PP^3$, is non-trivial
in that setting. For $g=4$ this is solved in \cite{BoLin}.
An alternative approach is to construct a 
classical curve over a non-archimedean field that tropicalizes to $\Gamma$.
See \cite[\S 7.3]{BBC} for first steps in that direction.

Example \ref{exp:shimuracurve} explored the Schottky locus in a
two-parameter family of Riemann matrices.
In the tropical setting, it is natural to intersect 
$ \mathfrak{H}_g^{\rm trop}$ with an
affine-linear space $L$ of symmetric matrices.
The intersection  $ \mathfrak{H}_g^{\rm trop} \cap L$
is a {\em spectrahedron}.
By the {\em Schottky locus of a spectrahedron}
we mean $ \mathfrak{J}_g^{\rm trop} \cap L$.
This is an infinite periodic polyhedral complex
inside the spectrahedron. For
quartic spectrahedra \cite{ORSV}, when $g=4$,
this locus has codimension one.

\begin{example}[The Schottky locus of a quartic spectrahedron] \rm
We consider the matrix
$$
\begin{small}
\! Q = 
 \begin{bmatrix}1589 - 2922 s + 960 t &   789 - 1322 s & -820 + 660 s - 1350 t & \!\! -820 + 3260 s +    2550 t \\ 789 - 1322 s &   1589 - 2922 s - 960 t \! &\!\!\! -820 + 3260 s - 2550 t \!& -820 + 660 s +  1350 t \\ -820 + 660 s - 1350 t & \!-820 + 3260 s - 2550 t \! &   1665 + 450 s + 3120 t & -25 - 2930 s \\ \!-820 + 3260 s +    2550 t & -820 + 660 s + 1350 t & -25 - 2930 s &   1665 + 450 s - 3120 t\end{bmatrix} \!.
 \end{small}
$$
Here $s$ and $t$ are parameters. This defines a plane $L$ in the
space of symmetric $4 \times 4$-matrices. The left diagram in  Figure \ref{fig:spec}
shows the hyperbolic curve $\{{\rm det}(Q)=0\}$.
The spectrahedron $ \mathfrak{H}_g^{\rm trop} \cap L$ is bounded by its inner oval.
The right diagram shows the second Voronoi decomposition.
The Schottky locus  $ \mathfrak{J}_g^{\rm trop} \cap L$ 
is a proper subgraph 
of its edge graph. 
It is shown in red. Note that the graph has infinitely many edges and regions.
\end{example}

\begin{remark} \rm
We described some computations in {\tt GAP} and in {\tt polymake}
that realize Algorithms~\ref{alg:TSD} and \ref{alg:TSR}.
The code for these implementations  is made available on our website~(\ref{eq:url}).
\end{remark}
\begin{figure}[h]
  \begin{center}
  \vspace{-0.1in} 
\includegraphics[width=7.5cm]{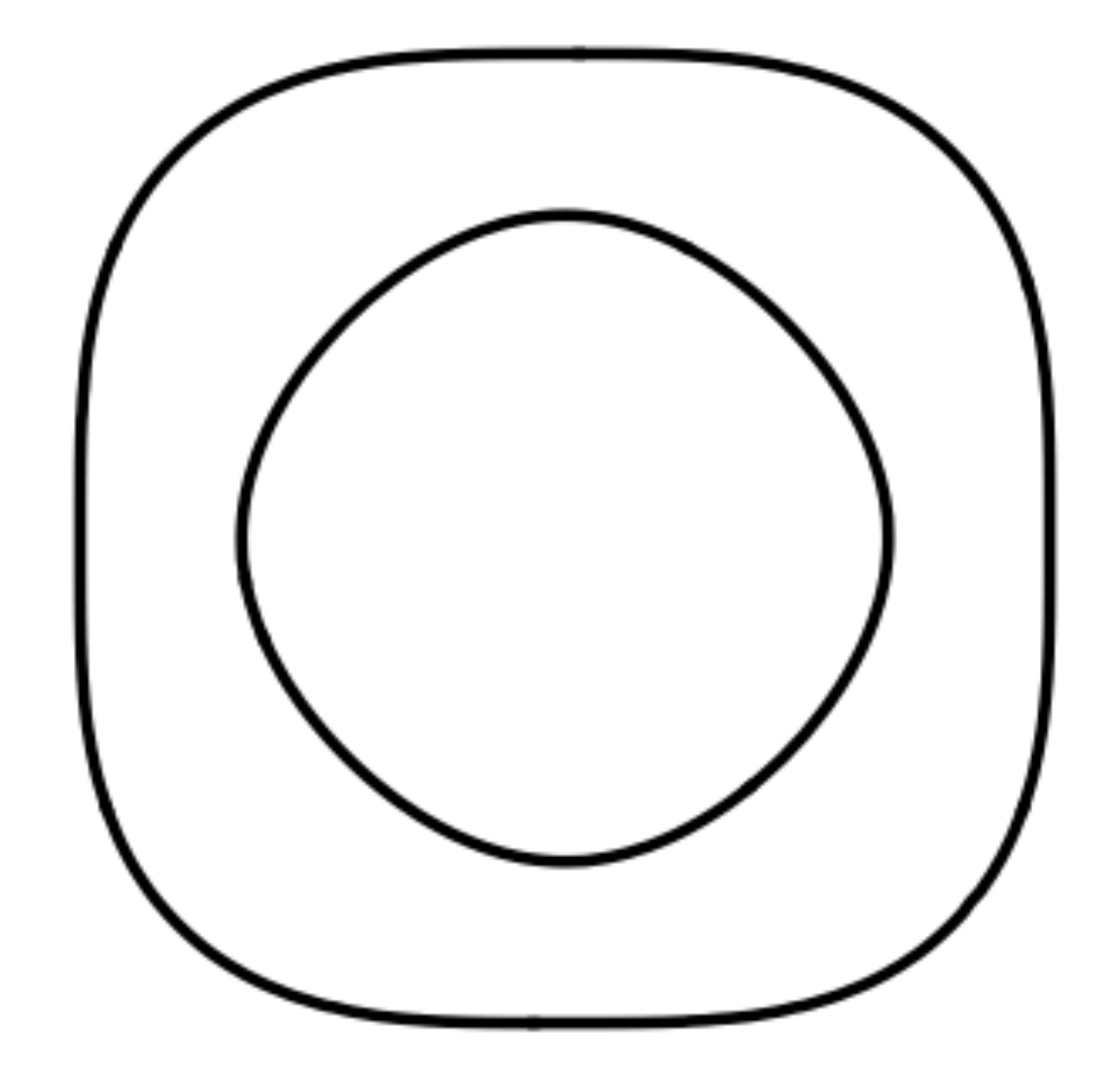}   \quad
 \includegraphics[width=7.8cm]{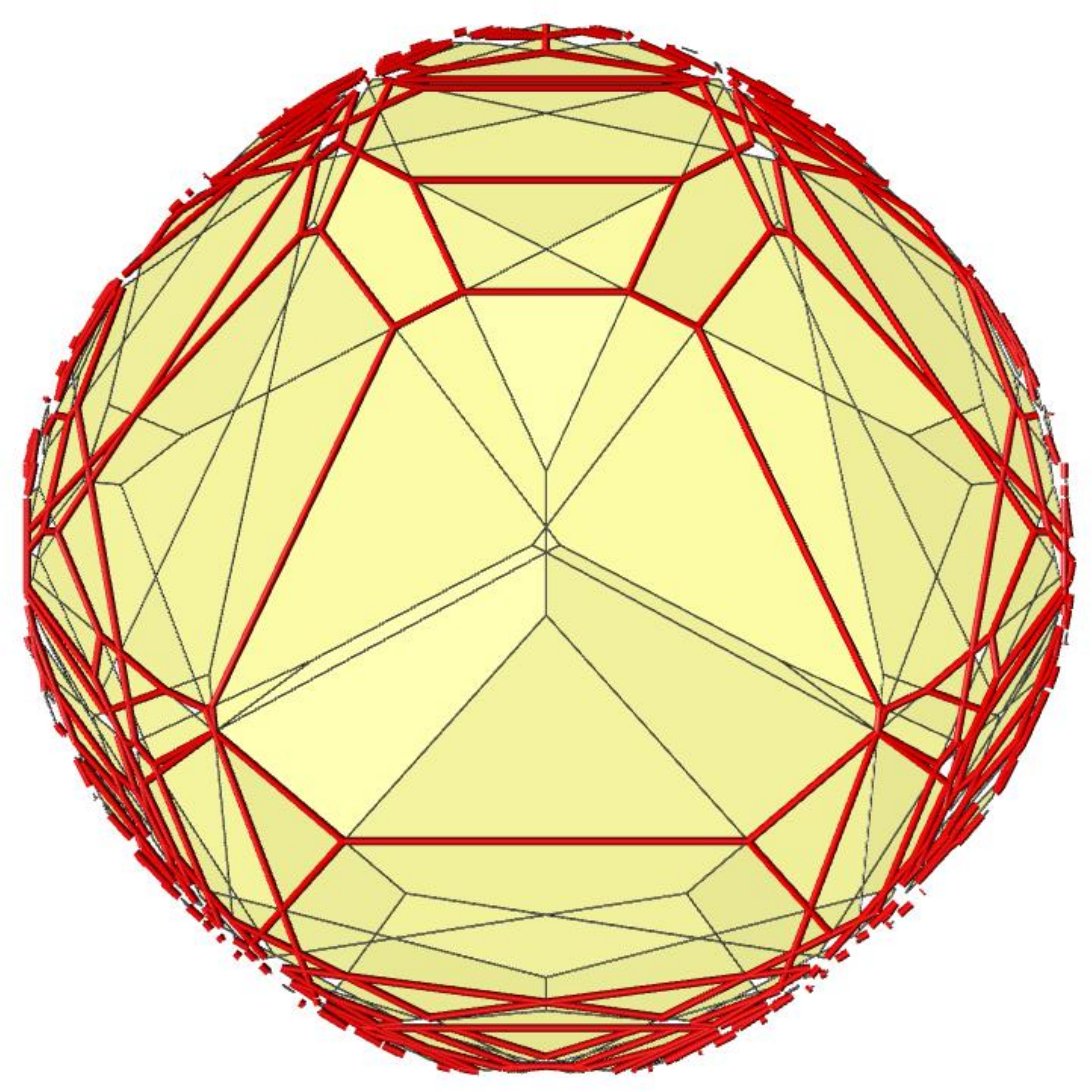}
\vspace{-0.3in}
\end{center}    \caption{\label{fig:spec}
A quartic spectrahedron (left) and its second Voronoi decomposition (right).
The Schottky locus of that spectrahedron consists of those edges that are highlighted in red.
}
 \end{figure}

\section{Tropical Meets Classical}
\label{sec:four}

In this section we present a second solution to the tropical Schottky problem.
It is new and different from the one in Section \ref{sec:tropical}, and it
links directly to  the classical solution in Section~\ref{sec:classic}.

Let $Q \in \mathfrak{H}_g^{\rm trop}$ be a positive definite matrix for arbitrary $g$. 
Mikhalkin and Zharkov \cite[\S 5.2]{mz} define the following analogue
to the Riemann theta function in the max-plus algebra:
\begin{equation}
\label{eq:troptheta}
\Theta(Q,x)\,\,:=\,\,\max_{\lambda\in\mathbb{Z}^g}\{\lambda^t Q x-\frac{1}{2} \lambda^t Q \lambda\}.
\end{equation}
This {\em tropical theta function} describes the asymptotic behavior of the 
classical Riemann theta function with Riemann matrix $t\cdot\tau$ when $t$ goes to infinity, as long as there are no cancellations.   This is made precise in Proposition \ref{lem:tropicallimit}. Here,
the real matrix $Q$ is the imaginary part of $\tau$. 

  Analogously, for $u\in\mathbb{Z}^g$, we define the \textit{tropical theta constant with characteristic} $u$ to be 
  \begin{equation}
  \label{eq:tropthetaconst}
 \Theta_{u}(Q)\,\,:=\,\,2\cdot\Theta(Q,\frac{u}{2})-\frac{1}{4} u^t Q u. 
\end{equation} 
  In the classical case, characteristics are vectors $m = (m',m'')$ in $\mathbb{Z}^{2g}$. But, only
  $u=m'$  contributes to the aforementioned asymptotics. 
  Note that $\Theta_u(Q)$ depends only on $u$ modulo~$2$. 

\begin{definition} \rm
  For any $v\in\mathbb{Z}^g$ consider the following signed sum of tropical theta constants:
\begin{equation}
\label{eq:vartheta}
 \vartheta_v(Q) \,\,\,:=\sum_{u\in(\mathbb{Z}/2\mathbb{Z})^g} \!\! (-1)^{u^t v}\cdot \Theta_u(Q) .
 \end{equation}
 The \textit{theta matroid} $M(Q)$ is the binary matroid represented by the collection of vectors 
 \begin{equation}
 \label{eq:magicmatroid}
\bigl\{v \in(\mathbb{Z}/2\mathbb{Z})^g:\, \vartheta_v(Q)\neq 0 \bigr\}.
 \end{equation}
\end{definition}

The tropical theta constants and the theta matroid are invariant under basis changes
$S\in\textnormal{GL}_g(\mathbb{Z})$. We have
$\, \vartheta_u(Q)=\vartheta_{S^{-1} u}(S^tQ S)\,$ for all 
$u\in\mathbb{Z}^g$, and therefore
$M(Q) = M(S^tQS)$.

Here is the promised new approach to the  Schottky problem.
If $Q$ lies in the tropical Schottky locus then
 $M(Q)$ is the desired cographic matroid
 and (\ref{eq:vartheta}) furnishes edge lengths.

\begin{theorem}\label{thm:recoverfromtheta}
 If $\,Q \in \mathfrak{J}_g^{\rm trop}$ then the matroid $M(Q)$ is cographic.
 In that graph, we assign the length
$\,2^{3-g} \cdot \vartheta_v(Q)\,$ to the edge labeled $v$.
The resulting metric graph has Riemann matrix~$Q$.
\end{theorem}

This says, in particular, that $\vartheta_v(Q)$ is non-negative when $Q $ comes from a metric graph.

\begin{proof}
Since $Q \in \mathfrak{J}_g^{\rm trop}$, there exists a unimodular matrix 
$B = (b_1,\ldots,b_m) \in \{-1,0,+1\}^{g \times m}$
and a diagonal matrix $D = {\rm diag}(\ell_1,\ldots,\ell_m)$ such that $Q = B D B^t 
= \sum_{i=1}^m \ell_i b_i b_i^t$. We claim 
\begin{equation}
\label{eq:weclaim}
   \Theta_u (Q)\,\,=\,\,-\,\frac{1}{4}\,\cdot \! \sum_{b_i^t u \textnormal{ is odd}} \!\!\ \ell_i \qquad \quad
    \hbox{for all $\,u\in \ZZ^g$}.
    \end{equation}
 Here the $\ell_i$ are positive real numbers. First, we note that
 $$
  \Theta_u(Q)\,\,=\,\,\max_{\lambda\in\mathbb{Z}^g} \,\bigl\{ -(\lambda+\frac{u}{2})^t Q(\lambda+\frac{u}{2})\,\bigr\}
\,    \,\,\leq \,
  \,\,\sum_{i=1}^m -\ell_i\cdot \min_{\lambda\in\mathbb{Z}^g}\biggl\{ \bigl(\,b_i^t\cdot (\lambda+\frac{u}{2})\,\bigr)^2 \biggr\}.
  $$
 If $\,b_i^t u\,$ is even, then $b_i^t\cdot (\lambda+\frac{u}{2})=0$ for some  $\lambda\in\mathbb{Z}^g$.
     Otherwise, the absolute value of $b_i^t\cdot (\lambda+\frac{u}{2})$ is at least $1/2$. 
     This shows that $\Theta_u(Q)\leq-\frac{1}{4}\cdot\sum_{b_i^t u\textnormal{ is odd}} \ell_i$.
To derive the reverse inequality, let $\,I = \bigl\{ i \,:\,
u_i \,\hbox{is odd} \bigr\} \subset \{1,\ldots,g\}$.
 By a result of Ghouila-Houri \cite{houri} on unimodular matrices,
 we can find $w\in\mathbb{Z}^g$ with $w_i=\pm1$ if $i\in I$ and $w_i=0$ otherwise, 
 such that $b_i^t\cdot w\in\{0,\pm1\}$ for all $1\leq i\leq m$. 
 The vector $\lambda_0 = \frac{1}{2}(w-u)$ lies in $\mathbb{Z}^g$.
 One checks that
  \[-(\lambda_0+\frac{u}{2})^t Q (\lambda_0+\frac{u}{2})
  \,=\,
  \sum_{i=1}^m -\ell_i\cdot (b_i^t\cdot (\lambda_0+\frac{u}{2}))^2
  \,=\,-\frac{1}{4}\sum_{i=1}^m \ell_i\cdot (b_i^t\cdot w)^2
  \,=\,-\frac{1}{4}\cdot\sum_{b_i^t u\textnormal{ is odd}} \!\!\ell_i.\] 
  Therefore, we also have $\,\Theta_u(Q)\geq-\frac{1}{4}\cdot\sum_{b_i^t u\textnormal{ is odd}} \ell_i$.
  This establishes the assertion in (\ref{eq:weclaim}).

We next claim that, under the same hypotheses as above, 
the function in (\ref{eq:vartheta})  satisfies
\begin{equation}
\label{eq:nextclaimthat}
  \vartheta_v(Q)\,\,=\,\,\,2^{g-3}\!\sum_{b_i\equiv v \,{\rm mod} \,2} \!\! \ell_i
 \qquad \hbox{for all}\,\, v \in \mathbb{Z}^g . 
\end{equation}
Indeed, substituting the right hand side of (\ref{eq:weclaim}) for $\Theta_u(Q)$ into (\ref{eq:vartheta}), we find that
  $$
 \vartheta_v(Q)\,\,=\,\,-\frac{1}{4}\cdot\sum_{u\in(\mathbb{Z}/2\mathbb{Z})^g}\sum_{b_i^t u \textnormal{ is odd}} 
 (-1)^{u^t v}\cdot  \ell_i \,\,=\,\,-\frac{1}{4}\cdot\sum_{i=1}^m \ell_i \cdot (|E_i|-|O_i|),
  $$
  where $E_i=\{u\in(\mathbb{Z}/2\mathbb{Z})^g: b_i^t u \,\,{\rm odd},\, u^t v \, \,{\rm even}\}$ and 
  $O_i=\{u \in(\mathbb{Z}/2\mathbb{Z})^g: b_i^t u \,\,{\rm odd}, \, u^t v \,\,{\rm odd}\}$. 
    If $b_i\equiv v \,{\rm mod}\,2$ then $E_i=\emptyset$ and $|O_i|=2^{g-1}$. Otherwise, $|E_i|=|O_i|=2^{g-2}$.
    This proves~(\ref{eq:nextclaimthat}).

Since $Q \in \mathfrak{J}_g^{\rm trop}$, this matrix comes from a graph $G$.
We may assume that $G$ has no $2$-valent vertices. This ensures
that any pair is independent in the cographic matroid of $G$.

The column $b_i$ of the matrix $B$ records the coefficients of the $i$-th edge
in a cycle basis of the graph $G$. The residue class of $b_i$ modulo $2$ is unique.
For $v \in \mathbb{Z}^g$ with $b_i \equiv v \,{\rm mod} \, 2 $, the sum in
(\ref{eq:nextclaimthat}) has only term $\ell_i$, and we have
$\ell_i = 2^{3-g}\vartheta_v(Q)$. If $v \in \mathbb{Z}^g$ is not congruent to $b_i$
for any $i$ then $\vartheta_v(Q) = 0$. This proves that the theta matroid
$M(Q)$ equals the cographic matroid of $G$, and the edge lengths $\ell_i$
are recovered from $Q$ by the rule in Theorem \ref{thm:recoverfromtheta}.
\end{proof}

By Theorem \ref{thm:recoverfromtheta}, the
non-negativity of $\vartheta_v(Q)$ is a necessary condition for $Q$ to be in 
 $\mathfrak{J}_g^{\rm trop}$.
 
\begin{example} \rm
For the matrix $Q$ in Example \ref{ex-decision}, we find $\vartheta_{0001}(Q)=-\frac{1}{2}$. 
Hence $\,Q \not\in \mathfrak{J}_4^{\rm trop}$.
\end{example}

This necessary (but not sufficient) condition translates into the following algorithm:

\begin{algorithm}[Tropical Schottky Recovery]   \label{alg:TSR2}
\underbar{Input}: $Q \in \mathfrak{J}_g^{\rm trop}$. \\
\underbar{Output}: A metric graph $\Gamma $ whose Riemann matrix $Q_\Gamma$ equals $Q$. \\
1. Compute the theta matroid $M(Q)$. It is cographic and determines a unique graph $G$. \\
2. Compute all edge lengths using the formula $\ell_i = 2^{3-g}\vartheta_v(Q)$. Set
 $\,D = {\rm diag}(\ell_1,\ldots,\ell_m)$. \\
 3. Output the metric graph $(G,D)$.  \\
4. (Optional) As in Algorithm \ref{alg:TSR}, find a basis $B$ such that $BDB^t = Q$.
\end{algorithm}

\begin{example} \rm
Let $Q$ be the matrix  in Example \ref{ex-recov}.
For each $u \in (\mathbb{Z}/2\mathbb{Z})^4$, we list
the theta constant $\Theta_u(Q)$, the weight $2^{-1}\vartheta_u(Q)$ 
and the label of the corresponding edge in Figure \ref{fig-graph1}:
\setcounter{MaxMatrixCols}{20}
$$
\begin{small}
\begin{matrix}
  u &  \! 0001 \! & \! 0010 \! & \! 0011 \!& \! 0100 \! & \! 0101 \! & \! 0110 \! & \! 0111 \! &
  \! 1000 \! & \! 1001 \! &\! 1010 \! &\! 1011 \! &\! 1100\! &\! 1101 \! &\! 1110 \! &\!1111 \smallskip \\
-\Theta_u &\frac{29}{4} & \frac{23}{4} &5& \frac{19}{4} & \frac{13}{2} & 7 & \frac{31}{4} & \frac{17}{4} & 9 & \frac{17}{2} & \frac{33}{4} & \frac{13}{2} & \frac{43}{4} & \frac{41}{4}& \frac{21}{2} \smallskip \\
   2^{-1}\vartheta_u  &9&7&9&8&2&0&4&12&0&0&0&0&2&0&3 \\
   {\rm Edge} & e_2 & e_1 & e_3 & e_6 & e_7 &-& e_8 & e_9 &-&-&-&- & e_4 &-& e_5    \\
\end{matrix}
\end{small}
$$
\end{example}

We now explain the connection between the classical 
and tropical theta functions. In particular, we will show how 
the process of tropicalization relates
Theorems \ref{thm:igusa} and \ref{thm:recoverfromtheta}.

In order to tropicalize the Schottky--Igusa modular form, we must study the order of growth of the theta constants when the entries of the Riemann matrix grow. This information is captured by the tropical theta constants.
 The following proposition makes that precise.

\begin{proposition} \label{lem:tropicallimit}
Fix $Q \in \mathfrak{H}_g^{\rm trop}$, and let $P(t)$ be any real symmetric
$g \times g$-matrix that depends on a parameter $t\in\RR$.
  For every $m\in(\ZZ/2\ZZ)^{2g}$ there is a constant $C \in\RR$ such that
\begin{equation}
\label{eq:thetaratio}
 0 \,\,\leq \,\,\frac{| \,\theta[m](P(t)+t\cdot \textnormal{i} Q, 0)\, |}{|\, \exp (t \cdot\pi\cdot  \Theta_{m'}(Q)) \,|}\leq
 \,\, C\,\, \quad  \hbox{for all $\,t\geq0$}.
\end{equation}
Moreover, we can choose $P(t)$ such  that the ratio above does not approach zero for $t\to\infty$.
\end{proposition}

  Here $ \theta[m](\tau,0)$ is the classical theta constant from
   (\ref{eqn:thetafunc}), and    $\Theta_{m'}(Q)$ is the tropical theta constant
   defined  in (\ref{eq:tropthetaconst}).
We use the notation $m = (m',m'')$ for vectors in $\mathbb{Z}^{2g}$ as in Section~\ref{sec:classic}.

\begin{proof}
Consider the lattice points $\lambda$ where the maximum in
(\ref{eq:troptheta}) for $x = m'/2$ is attained.
The corresponding summands in (\ref{eqn:thetafunc}) with $\lambda = n$ 
have the same asymptotic behavior as $\exp (t \cdot\pi  \Theta_{m'}(Q))$ for $t\to\infty$.
 The sum over the remaining exponentials tends to zero since it can be bounded by a sum of finitely many Gaussian integrals with variance going to zero for $t\to\infty$. We can choose the
 real symmetric matrix $P(t)$ 
 in such a way that no cancellation of highest order terms happens. Then the expression 
 in (\ref{eq:thetaratio}) is bounded away from zero.
 \end{proof}

\begin{remark}
 On the Siegel upper-half space $\mathfrak{H}_g$ we have an action by the symplectic group $\textnormal{Sp}_{2g}(\mathbb{Z})$. Two matrices from the same orbit under this action correspond to the same abelian variety. However their tropicalizations may vary drastically. Consider for example the case $g=1$: $\begin{pmatrix}
               0&-1\\                                                                                                                                                                                                                                                                                                                                                                                         
               1&k                                                                                                                                                                                                                                                                                                                                                                                        \end{pmatrix}\in\textnormal{Sp}_{2}(\mathbb{Z})$
 sends $\tau=\textnormal{i}$ to a complex number with imaginary part $\frac{1}{1+k^2}$. 
\end{remark}

We now assume that $g=4$.
For any subset $M\subset (\ZZ/2\ZZ)^{8}$ we write $M'=\{m':\, m\in M\}$ and similarly for $M''$.
The following lemma concerns the possible choices
 for Theorem \ref{thm:igusa}.

\begin{lemma}\label{cor:conj} For any azygetic triple $\{m_1,m_2,m_3\}$ and any matching
subgroup $N \subset (\mathbb{Z}/2\mathbb{Z})^8$,
\begin{itemize}
\item[(1)] there exist indices $1 \leq i < j \leq 3$ such that $(m_i+N)'=(m_j+N)'$, and \vspace{-0.1in}
\item[(2)] if $\dim N' =3$ and $(m_1+N)'=(m_2+N)'\neq (m_3+N)'$, then $m_1',m_2'\in N'$.
\end{itemize}
\end{lemma}

\begin{proof}
This purely combinatorial statement can be proved by exhaustive computation.
\end{proof}

For instance, consider the specific choice of $m_1,m_2,m_3,N$ made prior to Example \ref{exp:schottkycurve}.
This has $\dim N'=2$, and Lemma \ref{cor:conj} (1) holds with $i= 2$, $j= 3$.
If we exchange the first four coordinates with the last four coordinates, then 
$\dim N'=3$, $m_1',m_3'\in N'$ and $m_2'\not\in N'$.

Recall from Theorem \ref{thm:igusa} that a
 matrix $\tau\in\mathfrak{H}_4$ is in the Schottky locus if and only if $\pi_1^2+\pi_2^2+\pi_3^2-2(\pi_1\pi_2+\pi_1\pi_3+\pi_2\pi_3)$ vanishes. The tropicalization of this expression equals
 \begin{equation}
 \label{eq:tropicalschottkyigusa}
  \max_{i,j=1,2,3}( \pi^{\textnormal{trop}}_i+\pi^{\textnormal{trop}}_j), 
  \end{equation}
 where $\pi^{\textnormal{trop}}_i=\sum_{m\in m_i+N} \Theta_{m'}(Q)$
 is the tropicalization of the product (\ref{eq:pidef}), with $Q=
 {\rm im}(\tau)$.
 
 The {\em tropical Schottky--Igusa modular form} (\ref{eq:tropicalschottkyigusa})
defines a piecewise-linear convex function $\mathfrak{H}_4^{\rm trop} \rightarrow \RR$.
Its breakpoint locus  is the set of  Riemann matrices $Q$ for which the maximum in (\ref{eq:tropicalschottkyigusa})
  is attained twice. That set depends on our choice of $m_1,m_2,m_3,N$.
That choice is called {\em admissible} if $N\subset (\ZZ/2\ZZ)^8$ has rank three, the triple
$\{m_1,m_2,m_3\} \subset (\ZZ/2\ZZ)^8$ is azygetic, all elements of $m_i+N$ are even, 
\underbar{and} the group $N' \subset (\ZZ/2\ZZ)^4$ also has rank three.
We define the {\em tropical Igusa locus} 
in $\mathfrak{H}_g^{\rm trop}$ to be the intersection,
over all admissible choices $m_1,m_2,m_3,N$, of
the breakpoint loci of the tropical modular forms $(\ref{eq:tropicalschottkyigusa})$.

\begin{theorem}
\label{thm:admissible}
A matrix $Q \in \mathfrak{H}_4^{\rm trop}$ lies in the tropical Igusa locus
if and only if   $\vartheta_v(Q)\geq0$ for all $v\in\ZZ^4$.
That locus contains
the tropical Schottky locus $\mathfrak{J}_4^{\rm trop}\!$,
but they are not equal.
\end{theorem}

\begin{proof}
We are interested in how the maximum in (\ref{eq:tropicalschottkyigusa}) is attained.
By Lemma \ref{cor:conj} (1), after relabeling, $\,\pi^{\textnormal{trop}}_1=\pi^{\textnormal{trop}}_2$.
   The maximum is attained twice  if and only if $\pi^{\textnormal{trop}}_1\geq \pi^{\textnormal{trop}}_3$.
   By Lemma \ref{cor:conj} (2),    this is equivalent to 
\begin{equation}
\label{eq:inout}
\sum_{u\in N'} \Theta_{u}(Q)\,\,\geq \,\,\sum_{u\not\in N'} \Theta_{u}(Q). 
\end{equation}
Let $v $ be the non-zero vector in $(\ZZ/2\ZZ)^{4}$ that is orthogonal to $N'$. 
Then (\ref{eq:inout}) is equivalent to 
$$ \vartheta_v(Q) \,\,\, = 
\sum_{u\in(\ZZ/2\ZZ)^{4}}\!\!(-1)^{u^t v} \Theta_u(Q) 
\,\,\, \geq \,\,\, 0 . $$
This proves the first assertion, if we knew that
 every $v$ arises from some admissible choice.
 
We saw in Theorem \ref{thm:recoverfromtheta}
that $\vartheta_v(Q) \geq 0 $  for all $v$
whenever $Q \in \mathfrak{J}_4^{\rm trop}$.
Hence the tropical Schottky locus $\mathfrak{J}_4^{\rm trop}$ is contained in the tropical Igusa locus.
The two loci are not equal because the latter contains the 
{\em zonotopal locus} of $\mathfrak{H}_4^{\rm trop}$. This consists of matrices
$Q = BDB^t$ where $B$ represents any unimodular matroid, not necessarily
cographic. By  \cite[\S 4.4.4]{frank},  the second Voronoi decomposition
of $\mathfrak{H}_4^{\rm trop}$ has a non-cographic $9$-dimensional cone in its zonotopal locus.
It is unique modulo ${\rm GL}_4(\ZZ)$. We
verified that all $16$ tropical modular forms $\vartheta_v$ are non-negative on that cone.
This establishes the last assertion in Theorem \ref{thm:admissible}.

To finish the proof, we still need that every 
$v \in (\ZZ/2\ZZ)^{4} \backslash \{0\}$ is orthogonal to $N'$
for some admissible choice $m_1,m_2,m_3,N$.
By permuting coordinates,
  it suffices to show this~for 
  $$v \,\in \, \{(1,0,0,0)^t, (1,1,0,0)^t, (1,1,1,0)^t, (1,1,1,1)^t\}.$$ 
  For $v=(1,0,0,0)^t$ we take
$$
\begin{small}
 m_1=\begin{pmatrix}0 & 1 \\ 0 & 0 \\ 0 & 0 \\ 0 & 1     \end{pmatrix}\!,\,
m_2=\begin{pmatrix} 1 & 1 \\ 1 & 0 \\ 1 & 1 \\ 0 & 1    \end{pmatrix}\!,\,
m_3=\begin{pmatrix} 0 & 1 \\ 0 & 0 \\ 1 & 0 \\ 1 &  0  \end{pmatrix}\!,\,\,
  n_1=\begin{pmatrix} 0 & 1 \\ 1 & 0 \\ 0 &1 \\ 0 & 1  \end{pmatrix}\!,\,
  n_2=\begin{pmatrix} 0 & 1 \\ 0 & 0 \\ 1 & 0 \\ 1 & 1   \end{pmatrix}\!,\,
  n_3=\begin{pmatrix} 0 & 0 \\ 0 & 1  \\ 1 & 0\\ 0 & 0 \end{pmatrix}\!.
  \end{small}
$$ 
For $v=(1,1,0,0)^t$ we take
$$
\begin{small}
 m_1=\begin{pmatrix}0 & 1 \\ 0 & 0 \\ 0 & 0 \\ 0 & 1     \end{pmatrix}\!,\,
m_2=\begin{pmatrix} 1 & 1 \\ 1 & 0 \\ 1 & 1 \\ 0 & 0    \end{pmatrix}\!,\,
m_3=\begin{pmatrix} 1 & 0 \\ 0 & 0 \\ 0 & 0 \\ 0 &  1  \end{pmatrix}\!,\,\,
  n_1=\begin{pmatrix} 1 & 1 \\ 1 & 1 \\ 0 &1 \\ 1 & 0  \end{pmatrix}\!,\,
  n_2=\begin{pmatrix} 0 & 0 \\ 0 & 1 \\ 1 & 0 \\ 0 & 0   \end{pmatrix}\!,\,
  n_3=\begin{pmatrix} 0 & 0 \\ 0 & 1  \\ 0 & 0\\ 1 & 1 \end{pmatrix}\!.
\end{small}
$$ 
For $v=(1,1,1,0)^t$ we take
$$
\begin{small}
 m_1=\begin{pmatrix}1 & 0 \\ 0 & 0 \\ 0 & 0 \\ 1 & 0     \end{pmatrix}\!,\,
m_2=\begin{pmatrix} 0 & 1 \\ 0 & 1 \\ 0 & 1 \\ 1 & 0    \end{pmatrix}\!,\,
m_3=\begin{pmatrix} 0 & 0 \\ 0 & 0 \\ 0 & 0 \\ 1 &  0  \end{pmatrix}\!,\,\,
  n_1=\begin{pmatrix} 0 & 0 \\ 0 & 0 \\ 0 &0 \\ 1 & 1  \end{pmatrix}\!,\,
  n_2=\begin{pmatrix} 1 & 0 \\ 0 & 0 \\ 1 & 0 \\ 1 & 1   \end{pmatrix}\!,\,
  n_3=\begin{pmatrix} 1 & 0 \\ 1 & 0  \\ 0 & 0\\ 0 & 0 \end{pmatrix}\!.
  \end{small}
$$
For $v=(1,1,1,1)^t$ we take
$$
\begin{small}
 m_1=\begin{pmatrix}1 & 0 \\ 0 & 1 \\ 0 & 1 \\ 0 & 0     \end{pmatrix}\!,\,
m_2=\begin{pmatrix} 0 & 1 \\ 0 & 1 \\ 1 & 0 \\ 1 & 0    \end{pmatrix}\!,\,
m_3=\begin{pmatrix} 0 & 1 \\ 0 & 0 \\ 0 & 1 \\ 0 &  1  \end{pmatrix}\!,\,\,
  n_1=\begin{pmatrix} 1 & 0 \\ 1 & 1 \\ 0 &0 \\ 0 & 1  \end{pmatrix}\!,\,
  n_2=\begin{pmatrix} 1 & 0 \\ 0 & 1 \\ 0 & 1 \\ 1 & 0   \end{pmatrix}\!,\,
  n_3=\begin{pmatrix} 1 & 1 \\ 0 & 1  \\ 1 & 1\\ 0 & 0 \end{pmatrix}\!.
  \end{small}
$$ 
This completes the proof of Theorem \ref{thm:admissible}.
\end{proof}

We have shown that the tropicalization of the classical Schottky locus
satisfies the constraints coming from the tropical Schottky--Igusa modular forms
in (\ref{eq:tropicalschottkyigusa}). However, these constraints are not yet tight.
The tropical Igusa locus, as we have defined it, is strictly larger than the tropical
Schottky locus. It would be desirable to close this gap, at least for $g=4$.
One approach might be a more inclusive definition of which choices are ``admissible''.

\begin{question}
Can the tropical Schottky locus $\,\mathfrak{J}_4^{\rm trop}$ be cut out by
additional tropical modular forms, notably those obtained in  (\ref{eq:tropicalschottkyigusa})
by allowing choices $\,m_1,m_2,m_3,N\,$ with $\dim N' \leq 2 $?
\end{question}

The next question concerns arbitrary genus $g$.
We ask whether just computing the theta matroid $M(Q)$
solves the Tropical Schottky Decision problem. Note that we
did not address this subtle issue in Algorithm \ref{alg:TSR2} because
we had assumed that the input $Q $ lies in $ \mathfrak{J}_g^{\rm trop}$.

\begin{question}\label{qu:thetamatroid}
 Let $Q$ be a positive definite $g\times g$ matrix such that the 
 matroid $M(Q)$ is cographic with positive weights. Does this imply that $Q$ is in the tropical Schottky locus?
\end{question}

If the answer is affirmative then we can use Tutte's classical algorithm \cite{tutte} as
a subroutine for Schottky Decision. That algorithm can decide whether the
matroid $M(Q)$ is cographic.
%
%
We close with a question that pertains to classical Schottky Reconstruction  as in Section~\ref{sec:classic}.

\begin{question}\label{qu:riccardo}
How to generalize the results in {\rm \cite{bitangents}} from $g=3$ to $g=4$?
Is there a nice {\em tritangent matrix}, written explicitly in theta constants, for canonical curves of genus four?
\end{question}

\bigskip 
\medskip

\noindent
{\bf Acknowledgments.}
We thank Riccardo Salvati Manni for telling us about Kempf's article \cite{kempf}.
We also had helpful conversations with Christian Klein and Emre Sert\"oz.
Lynn Chua was supported by a UC Berkeley University Fellowship and the Max Planck Institute for Mathematics in the Sciences, Leipzig.
Bernd Sturmfels received funding from the US National Science
Foundation (DMS-1419018) and the Einstein Foundation Berlin.

\bigskip

\begin{small}


\footnotesize {
\noindent \bf Authors' addresses:} \smallskip \\
\noindent
Lynn Chua, UC Berkeley, {\tt chualynn@berkeley.edu} \\
Mario Kummer, TU Berlin, {\tt kummer@tu-berlin.de} \\
Bernd Sturmfels, MPI Leipzig,  {\tt bernd@mis.mpg.edu}  \
and UC Berkeley, \texttt{bernd@berkeley.edu}

\end{small}
\end{document}